\documentclass[a4paper]{amsart}
\usepackage{amsfonts,amsmath,amssymb,amscd,amstext,amsbsy,latexsym,url,tikz-cd}
\newtheorem{thm}{Theorem}[section]
\newtheorem{cor}[thm]{Corollary}
\newtheorem{lem}[thm]{Lemma}
\newtheorem{prop}[thm]{Proposition}
\theoremstyle{definition}
\newtheorem{defn}[thm]{Definition}
\theoremstyle{remark}
\newtheorem{rem}[thm]{Remark}
\newtheorem{exm}[thm]{Example}
\numberwithin{equation}{section}

\newcommand{\Max}{\operatorname{Max}}
\newcommand{\Maxl}{\operatorname{{Max}_\ell}}
\newcommand{\Maxr}{\operatorname{{Max}_r}}

\newcommand{\rann}{\operatorname{r.ann}}
\newcommand{\End}{\operatorname{End}}
\newcommand{\Hom}{\operatorname{Hom}}
\newcommand{\Ker}{\operatorname{Ker}}
\newcommand{\Img}{\operatorname{Im}}
\newcommand{\I}{\mathbb{I}}
\newcommand{\E}{\mathbb{E}}

\begin{document}

\title{Similar submodules of projective modules}%
\author{Alborz Azarang}%
\keywords{maximal submodule, projective module, maximal right ideal, endomorphism ring}%
\subjclass[2020]{16D40}%

\maketitle

\centerline{Department of Mathematics, Faculty of Mathematical Sciences and Computer,}
\centerline{Shahid Chamran University
of Ahvaz, Ahvaz-Iran} \centerline{azarang@scu.ac.ir}
\centerline{ORCID ID: orcid.org/0000-0001-9598-2411}
\begin{abstract}
We introduce a similarity relation between submodules of a module $M$ over a ring $R$, extending the classical notion of similarity for right ideals. Focusing on (faithfully) projective modules, we establish a sharp lower bound for the number of maximal submodules: if $N$ is a maximal submodule of $M$, then either $N$ is fully invariant or $N$ is similar to at least $1+|S|$ distinct maximal submodules, where $S$ is the eigenring of $N$; in particular, $|\Max(M)|\geq 1+|S|\geq 3$ in the latter case. For projective modules, we construct a canonical one-to-one map from $\Max(M)$ into $\Maxr(\End_R(M))$. When $M$ is faithfully projective and $\End_R(M)$ is right Artinian, we prove that $M$ has finite length and decomposes into a direct sum of local summands. Conversely, if $M$ is a projective right $R$-module with finite length, then $E_E$ has finite length with $\ell(E_E)\leq \ell(M_R)$; moreover, if $M$ is a faithfully projective $R$-module, then $\ell(E_E)=\ell(M_R)$; conversely, if $\ell(E_E)=\ell(M_R)$ holds, then $M$ is slightly compressible. These results are applied to obtain lower bounds on the number of maximal one-sided ideals that are not two-sided, with explicit consequences for matrix rings over infinite algebras.
\end{abstract}
\section{Introduction}
\subsection{Motivation and Results}
It is well-known that a ring $R$ has a unique maximal right ideal if and only if $R$ has a unique maximal left ideals; a ring with this property is called a local ring, see \cite{lam}. In fact, it is not hard to see that if $\mathfrak{m}$ is the unique maximal right ideal of $R$, then $\mathfrak{m}$ is the unique maximal left ideal of $R$, and consequently $\mathfrak{m}$ contains all proper right/left ideals of $R$, i.e., $\mathfrak{m}$ is the largest proper right/left ideal of $R$. For projective modules (which are natural generalizations of free modules) the analogue of this property was investigated by Ware \cite{ware}, who asked whether a projective module with a unique maximal submodule must be local (i.e., that unique maximal submodule contains every proper submodule).  He succeed in giving a positive answer for special classes of rings, such as left Noetherian and commutative rings; in particular, in a commutative case, he proved that for a projective $R$-module $P$ the following are equivalent:
\begin{enumerate}
  \item $P$ is a local module.
  \item $P$ has a unique maximal submodule.
  \item There exists a maximal ideal $\mathfrak{m}$ of $R$ such that $P\cong R_{\mathfrak{m}}$, as $R$-modules.
\end{enumerate}   
Moreover, if $M$ is the unique maximal submodule of $P$, then $P/M\cong R/\mathfrak{m}$, as $R$-modules.\\

Recently Sato settled the question in full generality, see \cite{sato}.\\

Despite this progress, a systematic framework for relating the collection of maximal submodules of a projective module to the ideal structure of its endomorphism ring has remained undeveloped. In particular, the well-established theory of similarity of right ideals (see \cite{cohnfir}) with its rich connections to idealizers, eigenrings, and module categories has not been extended to submodules in a way that preserves its structural power.\\

This paper fills that gap. We define a similarity relation between submodules that directly generalizes the classical one for right ideals. Our main results show that this relation tightly links the lattice of submodules of a projective module $M$ to the right-ideal lattice of $\End_R(M)$. Specifically, we prove that if $N$ is a maximal submodule of $M$, then either $N$ is fully invariant or the similarity class of $N$ contains at least $1+|\mathbb{E}(N)|$ distinct maximal submodules, where $\mathbb{E}(N)$ is the eigenring of $N$. Consequently, $|\Max(M)| \ge 3$ unless all maximal submodules are fully invariant.\\

When $M$ is projective, the assignment $N \mapsto \Hom_R(M,N)$ gives a one-to-one correspondence between $\Max(M)$ and a subset of $\Maxr(\End_R(M))$. This correspondence not only provides a new proof that a projective module is local if and only if its endomorphism ring is local, but also allows us to transfer finiteness conditions from the endomorphism ring to the module. For instance, if $M$ is faithfully projective and $\End_R(M)$ is right Artinian, then $M$ has finite length and decomposes as a direct sum of finitely many local summands.\\

In the final part of the paper, we apply these module-theoretic results to ring theory itself. We obtain lower bounds on the number of maximal one-sided ideals that are not two-sided, and show that for an infinite division subring $D \subseteq R$ and \(n>1\), the matrix ring $\mathbb{M}_n(R)$ always possesses infinitely many maximal left/right ideals that are not ideals. These corollaries illustrate how the theory of similar submodules can shed new light on classical questions in noncommutative algebra.\\

The paper is organized as follows. Section~2 collects the necessary preliminaries on retractable, faithfully projective, and slightly compressible modules, together with the basic theory of similarity for right ideals and submodules. Section~3 contains our main theorems on the abundance of maximal submodules and the correspondence with maximal right ideals of the endomorphism ring. It also deduces the structural consequences for faithfully projective modules with Artinian endomorphism rings. Finally, we transfer these results to the setting of rings, yielding the applications mentioned above.

\subsection{Notation and definitions}
All rings in this paper are unital with $1 \neq 0$. Unless stated otherwise, modules are right unital $R$-modules, subrings are unital, and homomorphisms preserve the identity. For a ring $T$ and a right $T$-module $M$, let $\Maxr(T)$, $\Maxl(T)$, $\Max(T)$, $\Max(M)$, $\ell(M)$ and $\rann_T(M)$, denote the set of all maximal right ideals of $T$, the set of all maximal left ideals of $T$, the set of all maximal ideals of $T$, the set of all maximal submodules of $M$, the length of $M$ (i.e., length of a composition series for $M$) and the right annihilator of $M$ in $T$, respectively. If $T$ is a ring, then $\mathbb{M}_n(T)$ denote the ring of all $n\times n$ matrices over $T$. For two right $T$-modules $M, M'$, $\Hom_T(M,M')$ denotes the set of all $T$-module homomorphisms from $M$ to $M'$; in particular, $\End_T(M) = \Hom_T(M,M)$ is the endomorphism ring of $M$. It is straightforward to verify that if $N$ is a submodule of $M$, then $\Hom_T(M,N)$ and $\Hom_T(N,M)$ are respectively a right ideal and a left ideal of $\End_T(M)$. A submodule $N$ of $M$ is called fully invariant if $\phi(N) \subseteq N$ for every $\phi \in \End_T(M)$; in this case $\Hom_T(M,N)$ is a two-sided ideal of $\End_T(M)$. Let $M$ be a right $T$-module and $E:=\End_T(M)$. $M$ is called indecomposable (resp. strongly indecomposable) if $E$ has no nontrivial idempotents (resp. if $E$ is a local ring). An idempotent element $e$ of a ring $T$ is called primitive (resp. local), if $eT$ is indecomposable (resp. local) as right $T$-module, see \cite{lam}. If $M$ is an $R$-module and $I$ an index set, we write $M^{(I)} := \bigoplus_{i \in I} M_i$ with $M_i = M$ for all $i \in I$. When $|I| = n$, we also use $M^{(n)}$ or simply $M^n$. For a right ideal $A$ of a ring $T$, $\I(A)$ denotes the idealizer of $A$ in $T$, i.e., the largest subring of $T$ in which $A$ is a two-sided ideal: $\I(A) = \{ t \in T \mid tA \subseteq A \}$. (The same notation is used for left ideals.) The quotient ring $\E_T(A) := \I(A)/A$ is called the eigenring of $A$. A ring $T$ is called right (resp. left) quasi-duo if every maximal right (resp. left) ideal of $T$ is two-sided; $T$ is quasi-duo if it is both left and right quasi-duo (see \cite{lamq}). For further notation and definitions we refer to  \cite{lam} and \cite{macrob}.

\section{Preliminaries}
\subsection{Retractable, slightly compressible and faithfully projective modules}

Recall that $\Hom_R(M,-)$ is a left exact functor, that is, whenever 

$$\begin{tikzcd}
0 \arrow[r] & A \arrow[r, "\phi"] & B  \arrow[r, "\psi"] & C 
\end{tikzcd}$$ 

is an exact sequence of $R$-modules, then we have the following exact sequence of abelian groups:
 
$$\begin{tikzcd}
0 \arrow[r] & \Hom_R(M,A) \arrow[r, "\phi^*"] & \Hom_R(M,B)  \arrow[r, "\psi^*"] & \Hom_R(M,C) 
\end{tikzcd}$$

where $\phi^*(\alpha)=\phi \circ \alpha$ and $\psi^*(\beta)=\psi\circ \beta$. In particular, a module $M$ is projective if and only if for every short exact sequence of $R$-modules 
$\begin{tikzcd} 0 \arrow[r] & A \arrow[r, "\phi"] & B  \arrow[r, "\psi"] & C \arrow[r] & 0 \end{tikzcd}$, the induced sequence
$$\begin{tikzcd}
0 \arrow[r] & \Hom_R(M,A) \arrow[r, "\phi^*"] & \Hom_R(M,B)  \arrow[r, "\psi^*"] & \Hom_R(M,C) \arrow[r] & 0 
\end{tikzcd}$$ 
is also exact.\\

Because $\Hom_R(M,-)$ is left exact and $\Hom_R(M,N)$ is a right ideal of $E=\End_R(M)=\Hom_R(M,M)$ for any submodules $N$ of $M$, we deduce that if $A\subseteq B$ are submodules of $M$, then $\Hom_R(M,A)\subseteq \Hom_R(M,B)$ are right ideals of $E$.\\

A (nonzero) right $R$-module $M$ is called retractable if $\Hom_R(M,N)\neq 0$ for every nonzero submodule $N$ of $M$, see \cite{khuri}; such a module is also termed slightly compressible, see \cite{smith}. Smith proved that over a right FBN ring $R$, every nonzero module is slightly compressible (i.e., $R$ is a retractable ring) if and only if $\Hom_R(M,R/P)\neq 0$, for each associated prime ideal $P$ of $M$. He further showed that every nonzero projective module over a commutative Noetherian ring is slightly compressible, and provided an example of a projective right ideal that is not slightly compressible \cite[Example 1.2]{smith}.\\

An $R$-module $M$ is called faithfully projective if the functor $\Hom_R(M,-)$ is a faithfully exact functor, see \cite{ishika}. For a projective $R$-module $M$ the following are equivalent, see \cite[Theorem 2.2]{ishika}.

\begin{enumerate}
  \item $M$ is a faithfully projective $R$-module.
  \item $\Hom_R(M,A)\neq 0$ for every nonzero $R$-module $A$. 
  \item $\Hom_R(M,R/I)\neq 0$ for each proper right ideal $I$ of $R$.
  \item $\Hom_R(M,A/\mathfrak{m})\neq 0$ for each maximal right ideal $\mathfrak{m}$ of $R$.
\end{enumerate}

Every faithfully projective module is projective; the converse holds whenever $R$ is commutative or $P$ is a finitely generated module, see \cite[Proposition 2.3]{ishika}. Moreover, for a projective $R$-module $M$, then the following are equivalent, see \cite[Proposition 2.4]{ishika}.

\begin{enumerate}
  \item $M$ is faithfully projective $R$-module.
  \item There exists a natural number $n$ such that $M^n$ has a direct summand isomorphic to $R$.
  \item There exists an infinite index set $I$ such that $M^{(I)}$ is free.
\end{enumerate}

These observations yield the following auxiliary result.

\begin{prop}\label{t1}
Let $M$ be a faithfully projective  $R$-module and $E=\End_R(M)$. If $E$ is a right Artinian (resp. Noetherian) ring, then $M$ is Artinian (resp. Noetherian).
\end{prop}
\begin{proof}
It suffices to show that whenever $A\subsetneq B$ are submodules of $M$, then we have $\Hom_R(M,A)\subsetneq \Hom_R(M,B)$. Since $B/A\neq 0$ and $M$ is a faithfully projective $R$-module, there exists a nonzero $R$-module homomorphism $g\in \Hom_R(M,B/A)$. Consider the diagram
$$
\begin{tikzcd}
 & M \arrow[ld, dashed, "f" above] \arrow[d, "g"] & \\
 B \arrow[r,"\pi"] & B/A \arrow[r] & 0
\end{tikzcd}
$$
where $\pi(b)=b+A$, for each $b\in B$. By projectivey of $M$, there exists $f\in \Hom_R(M,B)$ such that $\pi\circ f=g$. If $f\in \Hom_R(M,A)$, then $f(M)\subseteq A$ and hence $\pi(f(A))=0$, that is, $g=0$, a contradiction. Thus $\Hom_R(M,A)\subsetneq \Hom_R(M,B)$.
\end{proof}

\subsection{Similar right ideals and submodules}
Let $T$ be a ring and $I$ and $J$ right ideals of $T$. We say that $I$ and $J$ are similar right ideals, denoted by $I\sim J$, if $T/I\cong T/J$ as right $T$-modules. For $a\in T$ define $(I:a):=\{x\in T\ |\ xa\in I\}$. It is straightforward to check that $I\sim J$ if and only if there exists $c\in T$ such that $I+cT=T$ and $J=(I:c)$, see \cite[Proposition 1.3.6]{cohnfir}. For a right ideal $I$ of $T$ we have the ring isomorphism $\mathbb{I}(I)/I\cong \End_T(T/I)$, see \cite[P. 14, 1.11]{macrob}. Consequently, if $I\sim J$ then $\mathbb{I}(I)/I\cong \mathbb{I}(J)/J$ as rings. Thus, up to ring isomorphism, the ring $\mathbb{I}(I)/I$ is uniquely determined by similarity class; in particular $|\mathbb{I}(I)/I|=|\mathbb{I}(J)/J|$. In other words, the cardinality of $\mathbb{I}(I)/I$ is invariant under similarity. Now, let $M$ be a maximal right ideal of $T$. The following facts are useful:
\begin{enumerate}
\item A maximal right ideal $N$ of $T$ is similar to $M$ if and only if $N=(M:c)$ for some $c\in T\setminus M$.
\item If $N$ is a (maximal) right ideal of $T$ similar to $M$, then $N$ is an ideal of $T$ if and only if $N=M$, which in turn holds if and only if $M$ is an ideal of $T$.
\end{enumerate}
To see the previous facts, note that $(1)$ follows directly from \cite[Proposition 1.3.6]{cohnfir}, because obviously $M+cT=T$. For $(2)$, assume that $N$ and $M$ are similar right ideals and that $N$ is an ideal. Since $T/N\cong T/M$ as right $T$-modules, we have $N=\rann_T(T/N)=\rann_T(T/M)\subseteq M$. Because $N$ is a maximal right ideal, the inclusion $N\subseteq M$ forces $N=M$. Hence $(2)$ holds. Consequently, if $M$ is a maximal right ideal of $T$ that is not an ideal of $T$, and if $[M]$ denotes the set of all (maximal) left ideals of $T$ similar to $M$, then by $(2)$ we have $[M]\subseteq \Maxr(T)\setminus \Max(T)$, i.e., every (maximal) right ideal in $[M]$ fail to be a two-ideal. Clearly, similarity is an equivalence relation on the set of all right ideals of $T$, and in particular on $\Maxr(T)$. From $(1)$ and $(2)$, it is evident that for a maximal right ideal $M$, we have $|[M]|=1$ (i.e., $[M]=\{M\}$) if and only if $M$ is an ideal of $T$ (and therefore in this case $M\in \Max(T)\cap \Maxl(T)$).\\

Let $M$ be a right $R$-module, $N$ be an $R$-submodule of $M$ and $E=\End_R(M)$. It is not hard to see that $A:=\Hom_R(M,N)=\{\beta\in E\ |\ \beta(M)\subseteq N\}$ is a right ideal of $E$; $\mathbb{I}(N):=\{\beta\in E\ |\ \beta(N)\subseteq N\}$ is a subring of $E$ which is called the idealizer of $N$ in $M$ over $R$. The ring $\mathbb{E}(N):=\mathbb{I}(N)/A$ is called the eigenring of $N$ in $M$ over $R$ (note that obviously $A$ is an ideal of $\mathbb{I}(N)$), see \cite{cohnfir}. Analogous to the definition of similarity for right ideals, we define similar submodules as follows.

\begin{defn}
Let $N$ and $N'$ be $R$-submodules of a right $R$-module $M$ over a ring $R$. Then we say that $N$ is similar to $N'$ whenever $M/N\cong M/N'$ as $R$-modules. If $N$ is similar to $N'$, we write $N\sim N'$ (note that obviously $\sim$ is an equivalence relation on the set of submodules of $M$).
\end{defn}  

\begin{defn}
Let $N$ be a submodule of an $R$-module $M$ and $\beta\in E:=\End_R(M)$. Then $(N:\beta):=\{m\in M\ |\ \beta(m)\in N\}$. It is not hard to see that $(N:\beta)$ is an $R$-submodule of $M$.
\end{defn}

Using the previous notation and definitions, we have the following immediate lemma, whose proof is routine and has been omitted. 

\begin{lem}\label{t2}
Let $N$ be a submodule of an $R$-module $M$, $A:=\Hom_R(M,N)$, and $\beta\in E=\End_R(M)$. The following hold:
\begin{enumerate}
  \item $(N:\beta)=M$ if and only if $\beta\in A$.
  \item $N\subseteq (N:\beta)$ if and only if $\beta\in\mathbb{I}(N)$.
  \item $N\subseteq (N:\beta)\subsetneq M$ if and only if $\beta\in\mathbb{I}(N)\setminus A$.
  \item If $N\in \Max(M)$, then $N=(N:\beta)$ if and only if $\beta\in\mathbb{I}(N)\setminus A$.
\end{enumerate}
\end{lem}

\begin{lem}\label{t3}
Let $N$ be a submodule of an $R$-module $M$ and $\beta\in E=\End_R(M)$. If $N+\beta(M)=M$, then $(N:\beta)\sim N$.
\end{lem}
\begin{proof}
Define $\phi:M\longrightarrow M/N$ by $\phi(m)=\beta(m)+N$ for each $m\in M$. It is clear that $\phi$ is an $R$-module homomorphism. Since $N+\beta(M)=M$, we infer that $\phi$ is onto. It is not hard to see that $\Ker(\phi)=(N:\beta)$ and therefore $M/(N:\beta)\cong M/N$ as $R$-modules. Thus, $N\sim (N:\beta)$. 
\end{proof}

\section{Maximal submodules and projective modules}

We start this section by the following generalization of \cite[Proposition 1.3.6]{cohnfir} for projective modules.

\begin{prop}\label{t4}
Let $M$ be a projective $R$-module, $N$ and $N'$ be $R$-submodules of $M$. Then $N\sim N'$ if and only if there exists $\beta\in E=\End_R(M)$ such that $N'+\beta(M)=M$ and $N=(N':\beta)$.
\end{prop}
\begin{proof}
Assume that $N$ and $N'$ are similar and $\phi:M/N\longrightarrow M/N'$ is an $R$-module isomorphism. Let $i_1, i_2$ be the inclusion maps and $\pi_1,\pi_2$ be natural canonical maps, as in the following diagram. Since $M$ is a projective $R$-module, we conclude that there exists an $R$-module homomorphism $\alpha: M\longrightarrow M$ such that $\pi_2\circ\alpha=\phi\circ\pi_1$. Thus for each $m\in M$ we have $\alpha(m)+N'=\phi(m+N)$. Therefore $(\alpha(M)+N')/N'=\{\phi(m+N)\ |\ m\in M\}=M/N'$, because $\phi$ is an isomorphism. This immediately implies that $\alpha(M)+N'=M$. Now we prove that $N=(N':\alpha)$. To see this, note that $m\in (N':\alpha)$ if and only if $\alpha(m)\in N'$ if and only if $\phi(m+N)=\alpha(m)+N'=0$ if and only if $m\in N$, since $\phi$ is an isomorphism. The converse is evident by the previous lemma.   
$$
\begin{tikzcd}
0 \arrow[r] & N  \arrow[r, hook, "i_1"] & M \arrow[r, "\pi_1"] \arrow[d, dashed, "\alpha"] & \frac{M}{N} \arrow[r] \arrow[d, "\phi"] & 0 \\
0 \arrow[r] & N' \arrow[r, hook, "i_2"] & M  \arrow[r, "\pi_2"] & \frac{M}{N'} \arrow[r] & 0   
\end{tikzcd}
$$
\end{proof}

Now we present the following main result concerning the number of maximal submodules of an arbitrary module.

\begin{thm}\label{t5}
Let $M$ be an $R$-module, $E:=\End_R(M)$ and $N\in \Max(M)$. Then either $\mathbb{I}(N)=E$ (i.e., $N$ is fully invariant) or $N$ is similar to at least $1+|\mathbb{E}(N)|$ distinct maximal $R$-submodules of $M$. In particular, in the latter case, $|\Max(M)|\geq 1+|\mathbb{E}(N)|\geq 3$. 
\end{thm}
\begin{proof}
Assume that $N$ is not fully invariant in $M$, i.e., there exists $\alpha\in E\setminus\mathbb{I}(N)$ such that $\alpha(N)\nsubseteq N$. Thus $(N:\alpha)\neq N,M$. Hence by maximality of $N$ we conclude that $N+\alpha(M)=M$. From Lemma \ref{t3} we obtain that $(N:\alpha)\sim N$, i.e., $(N:\alpha)\in \Max(M)$. Clearly, for each $\beta\in \mathbb{I}(N)$ we have $\alpha+\beta\in E\setminus \mathbb{I}(N)$. Hence by the first part of the proof we conclude that $(N:\alpha+\beta)\sim N$ and thus $(N:\alpha+\beta)\in \Max(M)$. Now consider the map $\phi:\mathbb{I}(N)/A\longrightarrow \Max(M)$, which is defined by $\phi(\beta+A)=(N:\alpha+\beta)$, for each $\beta\in\mathbb{I}(N)$, where $A=\Hom_R(M,N)$ (It is not hard to see that $\phi$ is well-defined). By the first part of the proof note that $\Img(\phi)\subseteq \Max(M)\setminus\{N\}$. We claim that $\phi$ is one-one. Let $\beta_1,\beta_2\in\mathbb{I}(N)$ such that $\beta_1+A\neq \beta_2+A$, i.e., $\beta_1-\beta_2\in \mathbb{I}(N)\setminus A$. We show that $\phi(\beta_1+A)\neq \phi(\beta_2+A)$. Assume that $\phi(\beta_1+A)=\phi(\beta_2+A)$, i.e., $(N:\alpha+\beta_1)=(N:\alpha+\beta_2)$. Let $m\in(N:\alpha+\beta_1)=(N:\alpha+\beta_2)$, thus we deduce that $\alpha(m)+\beta_1(m)\in N$ and $\alpha(m)+\beta_2(m)\in N$. It follows that $\beta_1(m)-\beta_2(m)\in N$ and then $m\in (N:\beta_1-\beta_2)$. Since $\beta_1-\beta_2\in \mathbb{I}(N)\setminus A$ we have $(N:\beta_1-\beta_2)=N$. Hence $m\in N$ and consequently $(N:\alpha+\beta_1)=(N:\alpha+\beta_2)\subseteq N$. By maximality of $(N:\alpha+\beta_i)$, we deduce that $(N:\alpha+\beta_1)=(N:\alpha+\beta_2)=N$, which absurd by the first part of the proof (note $\alpha+\beta_i\in E\setminus \mathbb{I}(N)$). Thus $\phi$ is a one-one map from $\mathbb{E}(N)=\mathbb{I}(N)/A$ into $\Max(M)\setminus\{N\}$. It is clear that $N$ is similar to each $\phi(\beta+A)$, for any $\beta\in \mathbb{I}(N)$. Thus $N$ is similar to at least $1+|\mathbb{E}(N)|$ distinct maximal $R$-submodules of $M$. It is evident that $|\Max(M)|\geq 1+|\mathbb{E}(N)|$, and since $\mathbb{E}(N)$ is a ring we have $|\mathbb{E}(N)|\geq 2$, thus $|\Max(M)|\geq 3$.     
\end{proof}

For the next observation, we first need the following lemma.

\begin{lem}\label{t6}
Let $R$ be a $K$-algebra over a field $K$. Assume that $M$ is an $R$-module and $N$ is a proper submodule of $M$. Then $K$ naturally embeds into $\mathbb{E}(N)$.
\end{lem}
\begin{proof}
For each $k\in K$, define $\beta_k:M\longrightarrow M$ by $\beta_k(m)=mk$, for every $m\in M$. Since $R$ is a $K$-algebra, it is clear that $\beta_k\in E:=\End_R(M)$. Define $\phi:K\longrightarrow E$, by $\phi(k)=\beta_k$, for any $k\in K$. It is easy to see that $\phi$ is a ring monomorphism. Because $N$ is an $R$-submodule of $M$, we have $\beta_k(N)=Nk\subseteq N$; hence $\Img(\phi)\subseteq \mathbb{I}(N)$. Clearly $\Img(\phi)\cong K$ as rings, and therefore $\Img(\phi)$ is a field. Since $N$ is a proper submodule of $M$, we conclude that $A=\Hom_R(M,N)$ is a proper right ideal of $E$. So $A\cap \Img(\phi)=0$. Consequently, $\Img(\phi)$ embeds naturally into $\mathbb{E}(N)=\mathbb{I}(N)/A$. Thus $\mathbb{E}(N)$ contains a copy of $K$.
\end{proof}

We now obtain the following corollary.

\begin{cor}\label{t7}
Let $R$ be a $K$-algebra over a field $K$. Assume that $M$ is an $R$-module and $N\in \Max(M)$. Then either $\mathbb{I}(N)=\End_R(M)$ or $|\Max(M)|\geq 1+|K|$.
\end{cor}
\begin{proof}
By the previous lemma, $|\mathbb{E}(N)|\geq |K|$. The result now follows from Theorem \ref{t5}.
\end{proof}

In the next main result we see that the function $N\longmapsto \Hom_R(M,N)$ is a one-one function from $\Max(M)$ into $\Maxr(E)$, where $M$ is a projective $R$-module and $E=\End_R(M)$.

\begin{thm}\label{t8}
Let $M$  be a projective $R$-module and $N\in \Max(M)$. Then $A:=\Hom_R(M,N)\in \Maxr(E)$, where $E=\End_R(M)$. Moreover, the function $N\longmapsto \Hom_R(M,N)$ is a one-to-one function from $\Max(M)$ into $\Maxr(E)$. In other words, $|\Max(M)|\leq |\Maxr(E)|$.
\end{thm}
\begin{proof}
First we prove that $A$ is a maximal right ideal of $E$. It is clear that $A$ is a proper right ideal of $E$, because $1_M\notin A$. Hence assume that $B$ is a right ideal of $E$ which properly contains $A$. We prove that $B=E$. Let $\beta\in B\setminus A$, thus $\beta(M)\nsubseteq N$ and therefore by maximality of $N$ we deduce that $N+\beta(M)=M$. Define $\bar{\beta}:M\longrightarrow M/N$, by $\bar{\beta}(m)=\beta(m)+N$. Clearly $\bar{\beta}$ is an $R$-module homomorphism. Since $\beta(M)\nsubseteq N$ we infer that $\bar{\beta}\neq 0$. Also note that $\bar{\beta}$ is onto, because $M/N$ is a simple module. Hence we have the following diagram:
$$
\begin{tikzcd}
 & M \arrow[ld, dashed, "\alpha" above] \arrow[d, "\pi"] & \\
 M \arrow[r,"\bar{\beta}"] & M/N \arrow[r] & 0
\end{tikzcd}
$$
Since $M$ is a projective module we conclude that there exists an $R$-module homomorphism $\alpha\in E$ such that $\bar{\beta}\circ\alpha=\pi$. Hence for each $m\in M$ we have:
$$\beta(\alpha(m))+N=m+N=1_M(m)+N$$
Therefore $(1_M-\beta\alpha)(m)\in N$, for each $m\in M$. Thus $(1_M-\beta\alpha)(M)\subseteq N$, i.e., $1_M-\beta\alpha\in A\subseteq B$. Since $\beta\in B$ and $B$ is a right ideal of $E$, we deduce that $1_M\in B$ and thus $B=E$. Thus $A$ is a maximal right ideal of $E$.\\

Next we prove that the function $N\longmapsto \Hom_R(M,N)$ is a one-to-one function from $\Max(M)$ into $\Maxr(E)$. To see this, let $N$ and $N'$ be two distinct maximal submodules of $M$.
We prove that $\Hom_R(M,N)\neq \Hom_R(M,N')$. Obviously $M/N$ is a simple $R$-module, because $N$ is a maximal $R$-submodule of $M$. Since $N'\neq N$, we immediately obtain that the $R$-module homomorphism $g:N'\longrightarrow M/N$ which is defined by $g(x)=x+N$ is onto. Therefore we have the following diagram: 
$$
\begin{tikzcd}
 & M \arrow[ld, dashed, "f" above] \arrow[d, "\pi"] & \\
 N' \arrow[r,"g"] & M/N \arrow[r] & 0
\end{tikzcd}
$$
where $\pi(x)=x+N$, for each $x\in M$. By projectivy of $M$ we deduce that there exists an $R$-module homomorphism $f$ such that $gof=\pi$. Thus $f\in \Hom_R(M,N')$. If $f\in \Hom(M,N)$, then $\pi=0$, which is absurd. Thus $\Hom_R(M,N)\neq \Hom_R(M,N')$, as desird. The final part is evident now.
\end{proof}

Now we have the following interesting result of Ware.

\begin{thm}\label{t9}
Let $M$ be a projective $R$-module. Then $M$ is a local $R$-module if and only if $\End_R(M)$ is a local ring.
\end{thm}
\begin{proof}
By previous theorem it is obvious that if $\End_R(M)$ is local, then $M$ is local too. Conversely, assume that $M$ is a local $R$-module. Thus by \cite[Theorem 4.4]{sato}, there exists an idempotent $e$ of $R$ such that $P\cong eR$ as $R$-modules, $L=eJ(R)$ is the largest submodule of $M$ and $eR$ is indecomposable, i.e., $e$ is a primitive idempotent (note, in the statement and the proof of \cite[Theorem 4.3]{sato}, $e$ is called local whenever $eR$ is indecomposable, but in \cite[Proposition 21.8]{lam}, $e$ is called primitive whenever $eR$ is indecomposable). Since $L=eJ(R)$ is the unique maximal submodule of $M$, by \cite[Proposition 21.18]{lam} we deduce that $e$ is a local idempotent of $R$, i.e., $eRe$ is a local ring. Obviously $\End_R(M)\cong \End_R(eR)$ as rings and since $\End_R(eR)\cong eRe$ as rings (see \cite[Corollary 21.7]{lam}), we obtain that $\End_R(M)$ is a local ring.
\end{proof}

Next we obtain the following main result.

\begin{thm}\label{t10}
Let $M$ be a faithfully projective $R$-module and $E:=\End_R(M)$. If $E_E$ has finite length (i.e., $E$ is a right artinian ring), then $M$ has finite length. Moreover, there exist natural numbers $k$, $n_1,\ldots, n_k$ and local (projective and strongly indecomposable) submodules $M_1,\ldots,M_k$ of $M$ such that $M\cong M_1^{n_1}\oplus\cdots\oplus M_k^{n_k}$. In fact, there exist local idempotents $e_1,\ldots, e_k$ of $R$ such that $M\cong (e_1R)^{n_1}\oplus\cdots\oplus (e_kR)^{n_k}$. In particular, these decompositions of $M$ are unique.
\end{thm}
\begin{proof}
It is clear that by Proposition \ref{t1}, $M_R$ has finite length. Hence $M$ decompose to a finite direct sum of indecomposable submodules. Thus we may assume that there exist mutually nonisomorphic indecomposable submodules $M_1,\ldots,M_k$ of $M$ such that $M\cong M_1^{n_1}\oplus\cdots\oplus M_k^{n_k}$, for some natural numbers $k, n_1,\ldots,n_k$. It is clear that each $M_i$ is projective and has finite length. Thus each $M_i$ is strongly indecomposable by \cite[Theorem 19.17]{lam}. In other words, $\End_R(M_i)$ is a local ring and therefore $M_i$ is a local module by the previous theorem. For the next part, note that since each $M_i$ is a local projective module, then by the proof of the previous theorem there exists a local idempotent $e_i$ of $R$ such that $M_i\cong e_iR$. The final part is evident by Krull-Schmidt Theorem, see \cite[Corollary 19.22]{lam}.
\end{proof}

In the next main result we prove the converse of the previous theorem.

\begin{thm}
Let $M$ be a projective right $R$-module and $E:=\End_R(M)$. If $M$ has finite length, then $E_E$ has finite length with $\ell(E_E)\leq \ell(M_R)$. Moreover, if $M$ is a faithfully projective $R$-module, then $\ell(E_E)=\ell(M_R)$; conversely, if $\ell(E_E)=\ell(M_R)$, then $M$ is slightly compressible. 
\end{thm}
\begin{proof}
First we claim that whenever $A\subseteq B$ are submodules of $M$ and $A$ is maximal in $B$, then either $\Hom_R(M,A)=\Hom_R(M,B)$ or $\Hom_R(M,A)$ is maximal in $\Hom_R(M,B)$. To see this assume that $\Hom_R(M,A)\subsetneq \Hom_R(M,B)$. To show that $\Hom_R(M,A)$ is maximal in $\Hom_R(M,B)$, it suffices to prove that for each $\beta\in \Hom_R(M,B)\setminus \Hom_R(M,A)$, we have $\Hom_R(M,A)+\beta E=\Hom_R(M,B)$. Since $\beta\in \Hom_R(M,B)\setminus \Hom_R(M,A)$, we have $\beta(M)\nsubseteq A$ and therefore by maximality of $A$ in $B$ we infer that $A+\beta(M)=B$. Hence $\bar{\beta}(b)=\beta(b)+A$ is an $R$-module epimorphism from $M$ into $B/A$. Now, let $\alpha\in \Hom_R(M,B)$ and consider the $R$-module homomorphism $\bar{\alpha}(m)=\alpha(m)+A$ from $M$ into $B/A$. Consequently, we have the following diagram:  
$$
\begin{tikzcd}
 & M \arrow[ld, dashed, "f" above] \arrow[d, "\bar{\alpha}"] & \\
 M \arrow[r,"\bar{\beta}"] & B/A \arrow[r] & 0
\end{tikzcd}
$$
Since $M$ is a projective $R$-module, it follows that there exists an $R$-module homomorphism $f\in E$ such that $\bar{\beta}\circ f=\bar{\alpha}$. For each $m\in M$ we have $\beta(f(m))-\alpha(m)\in A$. In other words, $\beta\circ f-\alpha=\gamma\in \Hom_R(M,A)$. This means that $\alpha=\gamma+\beta f\in \Hom_R(M,A)+\beta E$. So $\Hom_R(M,B)=\Hom_R(M,A)+\beta E$, as desird. This immediately implies that whenever $0=A_0\subsetneq A_1\subsetneq A_2\subsetneq\cdots\subsetneq A_{n-1}\subsetneq A_n=M$ is a composition series for $M$, then from the chain
$$0=\Hom_R(M,A_0)\subseteq \Hom_R(M,A_1)\subseteq \Hom_R(M,A_2)\subseteq\cdots\subseteq \Hom_R(M,A_{n-1})\subseteq \Hom_R(M,A_n)=E$$
we obtain a composition series for $E$ with $\ell(E_E)\leq n=\ell(M)$. For the next part, note that when $M$ is a fiathfully projective $R$-module then by Proposition \ref{t1}, $\Hom_R(M,A_{i})\subsetneq \Hom_R(M,A_{i+1})$ and therefore the equality holds. Conversely, suppose that $\ell(E_E)=\ell(M)=n$ and consider a nonzero submodule $A$ of $M$. Now the chain $0\subset A\subseteq M$ can be inlarged to a composition series $0=A_0\subsetneq A_1\subsetneq A_2\subsetneq\cdots\subsetneq A_{n-1}\subsetneq A_n=M$, where for some $i\geq 1$ we have $A_i=A$. If $\Hom(M,A)=0$, then by the previous construction we obtain a composition series for $E_E$ of length $\leq n-1$, which is impossible. Consequently, $\Hom_R(M,A)\neq 0$ and therefore $M$ is slightly compressible.
\end{proof}

By the first part of the proof of the previous theorem we obtain the following quick corollary.

\begin{cor}
Let $M$ be a projective right $R$-module with finite length. Then for each simple $R$-submodule $A$ of $M$ either $\Hom_R(M,A)=0$ or $\Hom_R(M,A)$ is a minimal right ideal of $E$.
\end{cor}

We now study how similarity of submodules of a projective $R$-module $M$ relates to similarity of their corresponding right ideals in $E=\End_R(M)$. To begin, we compare the idealizer of a maximal submodule $N$ of $M$ with the idealizer of $\Hom_R(M,N)$ in $E$. 

\begin{prop}\label{t11}
Let $M$ be a projective $R$-module and $N\in \Max(M)$. Then $\mathbb{I}(N)=\mathbb{I}(A)$, where $A=\Hom_R(M,N)$. 
\end{prop}
\begin{proof}
First note that since $A$ is an ideal of $\mathbb{I}(N)$, we immediately conclude that $\mathbb{I}(N)\subseteq \mathbb{I}(A)$. Conversely, let $g\in\mathbb{I}(A)$, i.e., $gA\subseteq A$. Therefore for each $\alpha\in A$ we have $g\circ\alpha\in A$. We must prove that $g\in\mathbb{I}(N)$, i.e., $g(N)\subseteq N$. Thus assume that $g(N)\nsubseteq N$. By maximality of $N$ we obtain that $N+g(N)=M$. Define $\bar{g}:N\longrightarrow M/N$, by $\bar{g}(n)=g(n)+N$. It is clear that $\bar{g}\neq 0$, because $g(N)\nsubseteq N$. Since $M/N$ is a simple $R$-module we deduce that $\bar{g}$ is onto. Consequently, we have the following diagram:
$$
\begin{tikzcd}
 & M \arrow[ld, dashed, "\alpha" above] \arrow[d, "\pi"] & \\
 N \arrow[r,"\bar{g}"] & M/N \arrow[r] & 0
\end{tikzcd}
$$
where $\pi(m)=m+N$, for each $m\in M$. Since $M$ is a projective $R$-module, we infer that there exists an $R$-module homomorphism $\alpha\in \Hom_R(M,N)=A$ such that $\bar{g}\circ\alpha=\pi$. So, for each $m\in M$, we obtain that $g(\alpha(m))+N=m+N=1_M(m)+N$. This implies that $(1_M-g\circ\alpha)(M)\subseteq N$, i.e., $1_M-g\circ\alpha\in A$. Because $\alpha\in A$ and $g\in\mathbb{I}(A)$, we deduce that $g\circ\alpha\in A$ and then $1_M\in A$, which is imposible. Consequently $g(N)\subseteq N$, i.e., $g\in\mathbb{I}(N)$. Hence $\mathbb{I}(A)\subseteq \mathbb{I}(N)$ and therefore $\mathbb{I}(A)=\mathbb{I}(N)$. 
\end{proof}

In what follows, we show that similar submodules of a projective module $M$ correspond to similar right ideals in the endomorphism ring of $M$.

\begin{prop}\label{tsa1}
Let $M$ be a projective $R$-module and $E=\End_R(M)$. If $N$ and $N'$ are similar $R$-submodules of $M$, then $\Hom_R(M,N)$ and $\Hom_R(M,N')$ are similar right ideals of $E$.
\end{prop}
\begin{proof}
Since $N$ and $N'$ are similar submodules of $M$, by Proposition \ref{t4} we infer that there exists an $R$-module endomorphism $\beta$ of $M$ such that $N=N'+\beta(M)$ and $N=(N':\beta)$. We claim that:
\begin{enumerate}
\item $\Hom_R(M,N)=(\Hom_R(M,N'):\beta)$.
\item $E=\Hom_R(M,N')+\beta E$.
\end{enumerate}
For $(1)$, note that $\alpha\in \Hom_R(M,N)$ $\Longleftrightarrow$ $\alpha(m)\in N=(N':\beta)$ for each $m\in M$ $\Longleftrightarrow$ $\beta(\alpha(m))\in N'$ for each $m\in M$ $\Longleftrightarrow$ $\beta\circ\alpha\in \Hom_R(M,N')$ $\Longleftrightarrow$ $\alpha\in (\Hom_R(M,N'):\beta)$. Consequently $(1)$ holds.\\

To see $(2)$, it is clear that $\Hom_R(M,N')+\beta E\subseteq E$. Conversely, assume that $\alpha\in E$ and define $\bar{\alpha}(m)=\alpha(m)+N'$ for each $m\in M$, then clearly $\bar{\alpha}:M\rightarrow M/N'$ is an $R$-module homomorphism. Also note that $\bar{\beta}(m)=\beta(m)+N'$, for each $m\in M$, is an $R$-module epimorphism from $M$ onto $M/N'$, because $M=N'+\beta(M)$. Now since $M$ is a projective $R$-module, we infer that there exists an $R$-module homomorphism $f\in E$ such that $\bar{\beta}\circ f=\bar{\alpha}$. Consequently, for each $m\in M$ we have $(\beta\circ f-\alpha)(m)\in N'$. This implies that $\gamma:=\beta\circ f-\alpha\in \Hom_R(M,N')$. It follows that $\alpha=\gamma+\beta f\in \Hom_R(M,N')+\beta E$, as desired.\\

Thus by Proposition \ref{t4}, we conclude that $\Hom_R(M,N)$ and $\Hom_R(M,N')$ are similar right ideals of $E$.
\end{proof}

As the following example shows, the converse of the previous proposition does not hold.

\begin{exm}\label{tsa2}
We use a special case of \cite[Example 1.2]{smith}. Let $K$ be a field and $R=\begin{pmatrix}
  K & K  \\
  0 & K
\end{pmatrix}$, $A=\begin{pmatrix}
  K & K  \\
  0 & 0
\end{pmatrix}$ and $B=\begin{pmatrix}
  0 & K  \\
  0 & 0
\end{pmatrix}$. It is clear that for the idempotent $e=\begin{pmatrix}
  1 & 0  \\
  0 & 0
\end{pmatrix}$ we have $A=eR$ and therefore $A$ is a projective cyclic right $R$-module. Now by \cite[Example 1.2]{smith}, $\Hom_R(A,B)=0=\Hom_R(A,0)$ and therefore $\Hom_R(A,B)$ and $\Hom_R(A,0)$ are similar right ideals of $E=\End_R(A)$. But $B$ and $0$ are not similar $R$-submodules of $A$, because $A/0$ is not isomorphic to $A/B$ as right $R$-modules (note $A/B$ is a simple right $R$-module but $A$ is not simple).   
\end{exm}

Next we want to show that the converse of the previous proposition holds for faithfully projective modules. First we need some observation. We begin by the following lemma.

\begin{lem}\label{tsa3}
Let $M$ be an $R$-module and $\beta\in E=\End_R(M)$. Assume that $N$ and $N'$ are $R$-submodules of $M$, then the following hold:
\begin{enumerate}
\item If $E=\Hom_R(M,N')+\beta E$, then $M=N'+\beta(M)$.
\item If $\Hom_R(M,N)=(\Hom_R(M,N'):\beta)$, then $\Hom_R(M,N)=\Hom_R(M,(N':\beta))$.
\end{enumerate}
\end{lem}
\begin{proof}
$(1)$ Since $E=\Hom_R(M,N')+\beta E$, then $1_M=1_E=f+\beta g$, for some $f\in \Hom_R(M,N')$ and $g\in E$. Thus, for each $m\in M$ we have $m=f(m)+\beta(g(m))$. Consequently, $M\subseteq N'+\beta(M)$ and therefore the equality holds.\\
$(2)$ $h\in \Hom_R(M,N)\Longleftrightarrow h\in (\Hom_R(M,N'):\beta)\Longleftrightarrow \beta\circ h\in \Hom_R(M,N')\Longleftrightarrow \beta(h(M))\subseteq N'\Longleftrightarrow h(M)\subseteq (N':\beta)\Longleftrightarrow h\in \Hom_R(M, (N':\beta))$.
\end{proof}

An $R$-module $M$ is called a generator for an $R$-module $N$, if $\Hom_R(M,N)M=N$, that is, $N=\sum_{f\in \Hom_R(M,N)}\Img(f)$, see \cite{abyzov}. It is not hard to see that $M$ is a generator for $N$ if and only if $N$ is isomorphic to a quotient of $\bigoplus_{i\in I} M$, for some index set $I$. $M$ is called a generator for the category of right $R$-modules, if $M$ is a generator for ever $R$-module $N$. Equivalently, $M$ is a generator if and only if for any $R$-module $N$, and every $R$-submodules $A$ and $B$ of $N$, $\Hom_R(M,A)=\Hom_R(M,B)$ implies $A=B$. Clearly, if $M$ is a generator module, then $\Hom_R(M,A)\neq 0$ for ever nonzero $R$-module $A$; in particular, $M$ is a slightly compressible module. Next we need the following well-known lemma; we include its proof for completeness.

\begin{lem}\label{tsa4}
Let $M$ be an $R$-module. Then $M$ is a faithfully projective $R$-module if and only if $M$ is projective and generator.
\end{lem}
\begin{proof}
If $M$ is a projective and a generator module, then by the result of \cite{ishika}, mentioned in the preliminaries, $M$ is faithfully projective. Conversely, suppose $M$ is faithfully projective. As noted in the preliminaries, $M$ is projective, see \cite{ishika}. It remains to show that $M$ is a generator. Let $N$ be an $R$-module and set $K:=\Hom_R(M,N)M$. If $K=N$, we are done. Otherwise, consider the short exact sequence $\begin{tikzcd} 0 \arrow[r] & K \arrow[r, "i"] & N  \arrow[r, "\pi"] & N/K \arrow[r] &0  \end{tikzcd}$. By faithfulness, applying $\Hom_R(M,-)$ yield the exact sequence
$$\begin{tikzcd}
0 \arrow[r] & \Hom_R(M,K) \arrow[r, "i^*"] & \Hom_R(M,N)  \arrow[r, "\pi^*"] & \Hom_R(M,N/K) \arrow[r] &0  
\end{tikzcd}$$
where $i^*(f)=i\circ f$ and $\pi^*(g)=\pi\circ g$. For any $g\in \Hom_R(M,N)$ and for any $m\in M$, we have $g(m)\in \Img(g)\subseteq K$; Thus $\pi^*=0$ and consequently $\Hom_R(M,N/K)=0$. This contradicts the fact that $M$ is faithful projective. Therefore $K=N$ and  $M$ is a generator.
\end{proof}

From the previous discussion we immediately obtain the following corollary.
 
\begin{cor}\label{tsa5}
Let $M$ be a faithfully projective $R$-module, $E=\End_R(M)$, $N$ and $N'$ are $R$-submodules of $M$. Then $N$ and $N'$ are similar submodules of $M$ if and only if $\Hom_R(M,N)$ and $\Hom_R(M,N')$ are similar right ideals of $E$.
\end{cor}
\begin{proof}
If $N$ and $N'$ are similar in $M$, then we are done by Theorem \ref{tsa1}, because every faithfully projective module is projective. Conversely, assume that $\Hom_R(M,N)$ and $\Hom_R(M,N')$ are similar right ideals of $E$. Hence there exists $\beta\in E$ such that $E=\Hom_R(M,N')+\beta E$ and $\Hom_R(M,N)=(\Hom_R(M,N'):\beta)$. By Lemma \ref{tsa3}, we conclude that $M=N'+\beta(M)$ and  $\Hom_R(M,N)=\Hom_R(M,(N':\beta))$. The latter equality implies that $N=(N':\beta)$, for $M$ is a generator by Lemma \ref{tsa4}. Consequently, $N$ and $N'$ are similar by Proposition \ref{t4}. 
\end{proof}

In the proposition below, we establish another fact for fully invariant maximal submodules of projective modules.

\begin{prop}\label{t12}
Let $M$ be a projective $R$-module, $E=\End_R(M)$ and $N\in \Max(M)$. Then either $|\Max(M)|\geq 3$ or $\Hom_R(M,N)=(N:_Ex):=\{g\in E|\ g(x)\in N\}$ for every $x\in M\setminus N$.
\end{prop}
\begin{proof}
Assume that $|\Max(M)|\leq 2$. Hence by Theorem \ref{t5}, $\mathbb{I}(N)=E$, i.e., for each $f\in E$, we have $f(N)\subseteq N$. Since $M$ is a projective $R$-module and $N$ is a maximal submodule of $M$, by Theorem \ref{t8} we deduce that $A:=\Hom_R(M,N)$ is a maximal right ideal of $E$. From Theorem \ref{t11}, we obtain that $\mathbb{I}(A)=\mathbb{I}(N)$ and therefore $\mathbb{I}(A)=E$. Hence $A$ is a maximal ideal of $E$ (i.e., $E/A$ is a division ring). Now note that $A=\Hom_R(M,N)\subseteq (N:_Ex)\subsetneq E$ (note, $x\in M\setminus N$ and $1_M\notin (N:_Ex)$). We claim that $(N:_Ex)$ is a left ideal of $E$. Since $N$ is a submodule of $M$, it is obvious that $(N:_Ex)$ is a subgroup of $(E,+)$. Now suppose that $f\in E$ and $g\in (N:_Ex)$. It follows that $g(x)\in N$ and since $f\in E=\mathbb{I}(N)$, we infer that $f(g(x))\in f(N)\subseteq N$. Consequently $f\circ g\in (N:_Ex)$. Thus $(N:_Ex)$ is a proper left ideal of $E$ which contains $A$. This immediately implies that $A=(N:_Ex)$.
\end{proof}

For the next results, we need the following from \cite[Section 0.6]{cohnfir}. We present the proof for completeness. 

\begin{lem}\label{t13}
Let $M$ be an $R$-module, $N$ be an $R$-submodule of $M$ and $A:=\Hom_R(M,N)$. Then the ring $\mathbb{I}(N)/A$ embeds in $\End_R(M/N)$. Moreover, if in addition, $M$ is projective then $\mathbb{I}(N)/A\cong \End_R(M/N)$, as rings. 
\end{lem}
\begin{proof}
Define $\phi:\mathbb{I}(N)\longrightarrow \End_R(M/N)$ by $\phi(\alpha)=f_{\alpha}$, where $f_{\alpha}(m+N)=\alpha(m)+N$, for each $m\in M$. It is easy to check that $\phi$ is a ring homomorphism. Hence it suffices to show that $\Ker(\phi)=A$. Let $\alpha\in\mathbb{I}(N)$, then $\phi(\alpha)=0$ $\Longleftrightarrow$ $f_{\alpha}=0$ $\Longleftrightarrow$ $\alpha(M)+N=0$ $\Longleftrightarrow$ $\alpha(M)\subseteq N$ $\Longleftrightarrow$ $\alpha\in A$. Thus $\Ker(\phi)=A$ and therefore $\mathbb{I}(N)/A$ embeds in $\End_R(M/N)$. For the last part, suppose that $M$ is a projective and $\bar{\phi}$ is the $R$-module monomorphism which is induced by $\phi$. We prove that $\bar{\phi}$ is onto and therefore is a ring isomorphism. To this end, let $g\in \End_R(M/N)$. Consequently, we obtain the following diagram:
$$
\begin{tikzcd}
0 \arrow[r] & N  \arrow[r, hook, "i"] & M \arrow[r, "\pi"] \arrow[d, dashed, "\alpha"] & \frac{M}{N} \arrow[r] \arrow[d, "g"] & 0 \\
0 \arrow[r] & N \arrow[r, hook, "i"] & M  \arrow[r, "\pi"] & \frac{M}{N} \arrow[r] & 0   
\end{tikzcd}
$$
Since $M$ is a projective $R$-module, we infer that there exists an $R$-module homomorphism $\alpha:M\longrightarrow M$ such that $\pi\circ\alpha=g$. It follows that $g(m)=\alpha(m)+N$, for each $m\in M$. This implies that $\phi(\alpha)=g$. Thus $\bar{\phi}$ is onto, as desired.
\end{proof}

Now we have the following immediate corollaries.

\begin{cor}\label{t14}
Let $M$ be a projective module, $N\in \Max(M)$ and $A=\Hom_R(M,N)$. Then $\mathbb{I}(N)/A\cong \End_R(M/N)$ is a division ring.
\end{cor}

\begin{cor}\label{t15}
Let $M$ be a projective $R$-module. If $N$ and $N'$ are similar maximal $R$-submodules of $M$, then $\mathbb{E}(N)\cong\mathbb{E}(N')$.
\end{cor}

\begin{rem}\label{t16}
Let $M$ be a projective $R$-module, $N$ be a maximal $R$-submodule of $M$, $A=\Hom_R(M,N)$ and $E=\End_R(M)$. Then
$$\End_R(M/N)\cong\mathbb{I}(N)/A=\mathbb{E}(N)=\mathbb{E}(A)=\mathbb{I}(A)/A\cong \End_E(E/A).$$
\end{rem}


We close this paper with some observations and corollaries for the set of maximal right ideals of rings. It is clear that for any left ideal $I$ of a ring $T$ and any $c\in T$, we have $(I:c)=T$ if and only if $c\in I$. Moreover, $I\subseteq (I:c)$ holds if and only if $c\in\mathbb{I}(I)$. In particular, if $M$ is a maximal left ideal of $T$, these observations immediately imply that $(M:c)=M$ if and only if $c\in \mathbb{I}(M)\setminus M$. Consequently, $[M]=\{M\}\cup\{(M:c)\ |\ c\in T\setminus\mathbb{I}(M)\}$. Before starting our main result (which present a lower bound for $|[M]|$), we recall that, as noted earlier, for any $N\in [M]$, we have $|\mathbb{I}(M)/M|=|\mathbb{I}(N)/N|$. We now present the following theorem.

\begin{thm}\label{pt1}
Let $T$ be a ring and $M\in \Maxr(T)$ is not an ideal of $T$ (i.e., $M$ is not a maximal ideal of $T$). Then $|[M]|\geq |\mathbb{I}(M)/M|+1$. In particular, $|\Maxr(T)\setminus \Max(T)|\geq |\mathbb{I}(M)/M|+1$.
\end{thm}
\begin{proof}
Although we can prove with application of Theorem \ref{t5}, but we give a direct proof by the above comments. Let $R=\mathbb{I}(M)$ and $D=R/M$. Since $M$ is not an ideal of $T$ we infer that $R$ is a proper subring of $T$. Take a fixed element $x\in T\setminus R$. We claim that the function $c+M\longmapsto (M:x+c)$ is a one-one function from $D$ to $[M]\setminus \{M\}$. Since $x\in T\setminus R$ and $c\in R$ we deduce that $x+c\in T\setminus R$ and therefore $(M:x+c)\in [M]\setminus \{M\}$, by the previous observation. Hence it remain to show that the function is one-one. Hence assume that $a,b\in R$ and $(M:x+a)=(M:x+b)$, we must show that $a+M=b+M$, i.e., $a-b\in M$. Let $t\in (M:x+a)=(M:x+b)$, then we conclude that $tx+ta, tx+tb\in M$. It follows that $t(a-b)\in M$, i.e., $(M:x+a)\subseteq (M:a-b)$. Now if $a-b\notin M$, then we obtain that $a-b\in R\setminus M$, a fortiori $(M:a-b)=M$. This implies that $(M:x+a)\subseteq M$ and consequently by maximality of $(M:x+a)$ we deduce that $(M:x+a)=M$ which is absurd, because $x+a\in T\setminus R$. Thus $a-b\in M$, as desired.
\end{proof}

An immediate consequence of the previous theorem is that if $T$ is a ring and $\Maxl(T)\setminus \Max(T)\neq\emptyset$, then $|\Maxl(T)\setminus \Max(T)|\geq 3$. In other words, if a ring $T$ possesses a maximal left ideal that is not a two-ideal, then $T$ has at least three distinct maximal left ideals that are not two-ideals.

\begin{cor}\label{pt2}
Let $T$ be a $K$-algebra over a field $K$. Then either $T$ is a right quasi duo ring or $T$ has at least $|K|+1$ maximal right ideals which are not ideals of $T$ (i.e., $|\Maxr(T)\setminus \Max(T)|\geq |K|+1$).
\end{cor}
\begin{proof}
Assume that $T$ is not a right quasi duo ring and $M$ is a maximal right ideal of $T$ which is not an ideal of $T$. Since $K$ is a field which is contained in the center of $T$, we conclude that $K\subseteq R:=\mathbb{I}(M)$ and clearly $K\cap M=0$. Thus $K$ embeds in $R/M$ and hence we are done by Theorem \ref{pt1}.
\end{proof}

Hence we have the following immediate fact.

\begin{cor}\label{pt3}
Let $T$ be an algebra over an infinite field $K$. Then either $T$ is a right quasi duo ring or $\Maxr(T)\setminus \Max(T)$ is infinite (in fact, $|\Maxr(T)\setminus \Max(T)|\geq |K|$).
\end{cor}

Thus, if $T$ is a ring and $\Maxl(T)\setminus \Max(T)$ is finite, then for each $M\in \Maxl(T)\setminus \Max(T)$, we deduce that $\mathbb{I}(M)/M$ is a finite field. Moreover, if additionally, $T$ is a $K$-algebra over a field $K$ and $\Maxl(T)\setminus \Max(T)$ is finite, then $K$ must be a finite field (note that for any maximal left ideal $M$ of $T$, the field $\mathbb{I}(M)/M$ contains a copy of $K$).\\

Let $\Maxl(T)\setminus \Max(T)=\bigcup_{i\in I}[M_i]$ be the partition of $\Maxl(T)\setminus \Max(T)$ into similarity classes. By Theorem \ref{pt1}, we obtain that  $|\Maxl(T)\setminus \Max(T)|\geq |I|+\sum_{i\in I}|\mathbb{I}(M)/M|$.\\

It is clear from the above results that if $K$ is a field and $n>1$, then $|\Maxl(\mathbb{M}_n(K))|\geq |K|+1$ and $|\Maxr(\mathbb{M}_n(K))|\geq |K|+1$ (note, $\mathbb{M}_n(K)$ is a simple $K$-algebra and therefore each maximal left/right ideal of $\mathbb{M}_n(K)$ is not an ideal of $\mathbb{M}_n(K)$). In particular, if $K$ is an infinite field, then $\mathbb{M}_n(K)$ has at least $|K|$-many maximal left/right ideals (also see Remark \ref{pt5a}). Next we employ the notation and results of \cite{ston} to show that if $R$ is an infinite division ring and $n>1$, then $T=\mathbb{M}_n(R)$ possesses infinitely many maximal left/right ideals. Recall that if $M$ is a maximal left ideal of a ring $R$ and $T=\mathbb{M}_n(R)$, then for each $u\in R^n\setminus M^n$, the set $D(M,u):=\{X\in T\ |\ Xu\in M^n\}$ is a maximal left ideal of $T$; indeed, $T/D(M,u)\cong (R/M)^n$ as left $T$-modules. Moreover, $\Maxl(T)=\{D(M,u)\ |\ M\in \Maxl(R),\ u\in R^n\setminus M^n\}$, see \cite[Theorem 1.2]{ston}. In particular, when $R$ is a division ring, we obtain that $\Maxl(\mathbb{M}_n(R))=\{D(0,u)\ |\ 0\neq u\in R^n\}$.\\

Hence to show that $\Maxl(\mathbb{M}_n(R))$ is infinite for an infinite division ring $R$, it suffices to produce infinitely many nonzero vectors $u\in R^n$ for which $D(0,u)$ are distinct. By \cite[Proposition 2.3]{ston}, if $M$ is a maximal left ideal of a ring $R$, $S=\mathbb{I}(M)$ its idealizer of $M$ in $R$, and $u,v\in S^n$, then $D(M,u)=D(M,v)$ (in $T=\mathbb{M}_n(R)$) if and only if there exists $c\in S\setminus M$ such that $v\equiv uc\ \mod(M)$, i.e., $v-uc\in M^n$. In particular, for a division ring $R$ the only possible maximal left ideals is $M=0$, so that $S=R$; consequently, for $u,v\in R^n$, we have $D(0,u)=D(0,v)$ if and only if there exists $0\neq c\in R$ such that $v=uc$. In other words $v$ and $u$ are right-parallel vectors in the right $R$-module $R^n$. Clearly, if $R$ is an infinite division ring and $n>1$, then as right $R$-module $R^n$ contains infinitely many right non-parallel vectors. This observations yield the following quick corollary.

\begin{cor}\label{pt4}
Let $R$ be an infinite division ring and $n>1$. Then $T=\mathbb{M}_n(R)$ has infinitely many maximal left (resp. right) ideals.
\end{cor}

\begin{rem}\label{pt5a}
Let $R$ be a ring, $n>1$, and set $T=\mathbb{M}_n(R)$. Then it is easy to see that $\Maxl(T)\cap \Max(T)=\emptyset=\Maxr(T)\cap \Max(T)$. Moreover, the map $M\longmapsto M^t$ (where $M^t:=\{m^t\ |\ m\in M\}$ is the transpose of $M$), is a one-to-one correspondence from $\Maxl(T)$ onto $\Maxr(T)$. In particular, $|\Maxl(T)|=|\Maxr(T)|$.
\end{rem}

\begin{cor}\label{pt6b}
Let $R$ be an infinite ring. Then the following hold.
\begin{enumerate}
  \item If $R$ is a prime ring with $J(R)=0$, then either $R$ is a division ring or $\Maxl(R)$ and $\Maxr(R)$ are infinite.
  \item If $R$ is a left/right primitive ring, then either $R$ is a division ring or $\Maxl(R)$ and $\Maxr(R)$ are infinite. In particular, if $R$ is a one-sided primitive ring but is not a primitive ring, then $\Maxl(R)$ and $\Maxr(R)$ are infinite.
  \item If $R$ is a simple ring, then either $R$ is a division ring or $\Maxl(R)$ and $\Maxr(R)$ are infinite.
\end{enumerate}
\end{cor}
\begin{proof}
It suffices to prove $(1)$, the remaining statements follow immediately from $(1)$. Assume that $R$ is a prime ring with $J(R)=0$, and, for instance let $\Maxr(R)=\{\mathfrak{m}_1,\ldots,\mathfrak{m}_k\}$ is finite. Since $J(R)=0$, then $R$ embeds into $R/\mathfrak{m}_1\times\cdots\times R/\mathfrak{m}_k$. Thus $R$ is a right Artinian ring, hence $R$ is a semisimple ring. Because $R$ is a prime ring, it follows that $R\cong\mathbb{M}_n(D)$ for some natural number $n$ and some division ring $D$. The ring $D$ is infinite, because $R$ is infinite. If $n>1$, then by Corollary \ref{pt4}, $R$ possesses infinitely many maximal right ideals, a contradiction. Thus $n=1$ and $R$ is a division ring, as desired. 
\end{proof}

\begin{rem}\label{pt7c}
Let $n>1$ be a natural number, $R$ a ring and $T=\mathbb{M}_n(R)$. It is not hard to see that if $R$ has infinitely many maximal right ideals, then $T$ has infinitely many maximal right ideals (and, by Remark \ref{pt5a}, also infinitely many maximal left ideals). Likewise, if $R$ has infinitely many maximal ideals, then $T$ has infinitely many maximal left/right ideals. Observe that $\Maxl(R)\cap \Max(R)=\Maxr(R)\cap \Max(R)$. Now, if $T$ has only finitely many maximal right (or left) ideals, then from the above observations we conclude that $\Maxr(R)$, $\Maxl(R)$ and $\Max(R)$ are finite. Hence we may write:
\begin{enumerate}
  \item $\Maxr(R)=\{M_1,\ldots,M_k,P_1,\ldots,P_r\}$.
  \item $\Max_\ell(R)=\{M_1,\ldots,M_k,Q_1,\ldots,Q_s\}$.
  \item $\Maxr(R)=\{M_1,\ldots,M_k,N_1,\ldots,N_m\}$.
\end{enumerate}
Moreover, if $P$ is a right (or left) primitive ideal of $R$ and $R/P$ is infinite, then Corollary \ref{pt6b} implies that $R/P$ has infinitely many maximal right/left ideals, impossible under finiteness assumption. Therefore $R/P$ is a finite prime ring, and hence is a simple ring, that is, $P$ is a maximal ideal of $R$.\\

Therefore by Corollary \ref{pt6b} and Theorem \ref{pt1}, we obtain:
\begin{enumerate}
  \item For each $i$, the ring $R/M_i$ is a finite field.
  \item For each $i$, the ring $R/N_i$ is a finite simple ring of the form $\mathbb{M}_{n_i}(E_i)$, where $E_i$ is a finite field.
  \item For each $j$, the rings $\mathbb{I}(P_j)/P_j$ and $\mathbb{I}(Q_j)/Q_j$ are finite simple ring.
\end{enumerate}
In particular, $R/J^*(R)$ is a finite ring, where $J^*(R)=\bigcap_{M\in \Max(R)} M$.
\end{rem}

From the previous remark we now derive the following generalization of Corollary \ref{pt4}.

\begin{cor}
Let $D$ be an infinite division ring contained as a subring of a ring $R$. Then for every $n>1$, the ring $\mathbb{M}_n(R)$ has infinitely many maximal left/right ideals. Equivalently, any ring $T$ that contains  $\mathbb{M}_n(D)$ as a subring has infinitely many maximal left/right ideals.
\end{cor}

\vspace{0.5cm}
\centerline{\Large{\bf Declaration of Interest Statement}}
The authors declare that they have no known competing financial interests or personal relationships that could have appeared to influence the work reported in this paper.
\vspace{0.5cm}

\vspace{0.5cm}
\centerline{\Large{\bf Acknowledgement}}
The author is grateful to the Research Council of Shahid Chamran University of Ahvaz (Ahvaz-Iran) for
financial support (Grant Number: SCU.MM1404.721)
\vspace{0.5cm}


\end{document}